\documentclass[reqno]{amsart}

\usepackage{amssymb,amsfonts,amsmath,mathdots,amsthm}
\usepackage{float}
\usepackage{tikz-cd}
\usepackage{lmodern}
\usepackage{mathrsfs}
\usepackage[sc]{mathpazo}
\DeclareMathAlphabet{\mathfrak}{U}{euf}{m}{n}
\usepackage{mathtools}  
\usepackage{calligra}
\usepackage{xcolor}
\usepackage[hidelinks]{hyperref}
\usepackage{accents}
\usepackage{subcaption}
\usepackage{dashbox}
\usepackage{enumitem}
\usepackage{tabu}

\hypersetup{colorlinks=true,
	linkcolor=magenta,
	citecolor=magenta}

\setlist[itemize]{leftmargin=*}

\newcommand{\genu}{\alpha}
\newcommand{\genw}{u}

\newcommand{\geneidem}{e}

\usepackage{relsize}
\usepackage{etoolbox,fancyhdr}
\usepackage{indentfirst}  

\usetikzlibrary{arrows,arrows.meta,shapes.misc,decorations.markings,positioning,patterns}

\let\oldproofname=\proofname
\renewcommand{\proofname}{\bfseries\textup{\oldproofname}}

\mathtoolsset{showonlyrefs} 

\DeclareMathOperator{\Z}{\mathbb{Z}}
\DeclareMathOperator{\C}{\mathbb{C}}
\DeclareMathOperator{\Q}{\mathbb{Q}}

\DeclareMathOperator{\R}{\mathbb{R}}

\DeclareMathOperator{\Tr}{Tr}
\DeclareMathOperator{\Res}{Res}

\DeclareMathOperator{\Hom}{Hom}

\theoremstyle{plain}

\newtheorem{prop}{Proposition}[section]

\makeatother
\title{$C_{2^n}$-equivariant rational stable stems and characteristic classes}
\author{Nick Georgakopoulos}

\begin{document}

	\begin{abstract}In this short note, we compute the rational $C_{2^n}$-equivariant stable stems and give minimal presentations for the $RO(C_{2^n})$-graded Bredon cohomology of the equivariant classifying spaces $B_{C_{2^n}}S^1$ and $B_{C_{2^n}}\Sigma_2$ over the rational Burnside functor $A_{\Q}$. We also examine for which compact Lie groups $L$ the maximal torus inclusion $T\to L$ induces an isomorphism from $H^*_{C_{2^n}}(B_{C_{2^n}}L;A_{\Q})$ onto the fixed points of $H^*_{C_{2^n}}(B_{C_{2^n}}T;A_{\Q})$ under the Weyl group action. We prove that this holds for $L=U(m)$ and any $n,m\ge 1$ but does not hold for $L=SU(2)$ and $n>1$.
	\end{abstract}

	\maketitle{}
	
	\tableofcontents
	\section{Introduction}\label{Intro}
	
	This note is the followup to \cite{Rat1}. We start by computing the $C_{2^n}$-equivariant rational stable stems; this is done in section \ref{RationalStems}. While the method employed here is the one used in \cite{Rat1} and goes back to \cite{GM95}, the result is quite a bit more complicated to state and requires the notation set up in sections \ref{RationalMack} and \ref{EulerAndOr}.

	We then attempt to generalize the results in \cite{Rat1} to groups $C_{2^n}$. In \cite{Rat1}, we obtained minimal descriptions of the $C_2$-equivariant Chern, Pontryagin and symplectic characteristic classes associated with genuine (Bredon) cohomology using coefficients in the rational Burnside Green functor $A_{\Q}$. 
	The idea was based on the maximal torus isomorphism: if $L$ is any one of $U(m), Sp(m), SO(m), SU(m)$, $T$ is a maximal torus in $L$ and $W$ is the associated Weyl group then the inclusion $B_{C_2}T\to B_{C_2}L$ induces an isomorphism $H^{\bigstar}_{C_2}(B_{C_2}L;A_{\Q})\to H^{\bigstar}_{C_2}(B_{C_2}T;A_{\Q})^W$. We then computed $H^{\bigstar}_{C_2}(B_{C_2}T;A_{\Q})$ from $H^{\bigstar}_{C_2}(B_{C_2}S^1;A_{\Q})$ and the Kunneth formula, which reduced us to the algebraic problem of computing a minimal presentation of the fixed points $H^{\bigstar}_{C_2}(B_{C_2}T;A_{\Q})^W$.
	
	In section \ref{RationalCharacteristic}, we show that the maximal torus isomorphism holds for all groups $G=C_{2^n}$ and $L=U(m)$, but does not hold for $G=C_{2^n}$ and $L=SU(2)$ when $n>1$. We also compute the Green functor $H^{\bigstar}_G(B_GS^1;A_{\Q})$ which turns out to be algebraically quite a bit more complex compared to the $C_2$ case of \cite{Rat1}. For that reason, we do not attempt to follow the program in \cite{Rat1} and get minimal descriptions of $H^{\bigstar}_G(B_GU(m);A_{\Q})$ from the maximal torus isomorphism.

	\subsection*{Acknowledgment}We would like to thank Peter May for his numerous editing suggestions, including the idea to split off this paper from \cite{Rat1}.

	\section{Rational Mackey functors}\label{RationalMack}
	
	The rational Burnside Green functor $A_{\Q}$ over a group $G$ is defined on orbits as $G/H\mapsto A(H)\otimes \Q$ where $A(H)$ is the Burnside ring of $H$. A rational $G$-Mackey functor is by definition an $A_{\Q}$ module.\medbreak
	
	We shall use $G$-equivariant \emph{unreduced} co/homology in $A_{\Q}$ coefficients. So if $X$ is an unbased $G$-space, $H_{\bigstar}^G(X)$ is the rational $G$-Mackey functor defined on orbits as
	\begin{equation}
	H_{\bigstar}^G(X)(G/H)=[S^{\bigstar},X_+\wedge HA_{\Q}]^H
	\end{equation}
	where $HA_{\Q}$ is the equivariant Eilenberg-MacLane spectrum associated to $A_{\Q}$ and the index $\bigstar$ ranges over the real representation ring $RO(G)$. \medbreak
	
	We warn the reader of differing conventions that can be found in the literature: $H_{\bigstar}(X)$ is sometimes used to denote the \emph{reduced} homology Mackey functor (the group $G$ being implicit), with $H_{\bigstar}^G(X)$ denoting the value of this Mackey functor on the top level (i.e. the $G/G$ orbit). In this paper, $H_{\bigstar}^G(X)$ always denotes the Mackey functor and $H_{\bigstar}^G(X)(G/G)$ always denotes the top level. This convention also applies when $\bigstar=*$ ranges over the integers, in which case $H_*(X)$ denotes the nonequivariant rational homology of $X$.\medbreak
	
	All these conventions apply equally for cohomology $H^{\bigstar}_G(X)$.\medbreak
	
	The $RO(G)$-graded Mackey functor $H^G_{\bigstar}(X)$ is a module over the homology of a point  $H^{G}_{\bigstar}:=H^G_{\bigstar}(*)$. This Green functor agrees with the equivariant rational stable stems:
	 \begin{equation}
	 \pi_{\bigstar}^GS\otimes \Q=H^G_{\bigstar}
	 \end{equation}	
	
	Two facts about rational Mackey functors that we shall liberally use (\cite{GM95}):
	\begin{itemize}
		\item All rational Mackey functors are projective and injective, so we have the Kunneth formula:
		\begin{equation}
		H_{\bigstar}^G(X\times Y)=H_{\bigstar}^G(X)\boxtimes_{H_{\bigstar}^G}H_{\bigstar}^G(Y)
		\end{equation}
		and duality formula:
		\begin{equation}
		H_G^{\bigstar}(X)=\Hom_{H_{\bigstar}^G}(H_{\bigstar}^G(X),H_{\bigstar}^G)
		\end{equation}
		\item For a $G$-Mackey functor $M$ and a subgroup $H$ of $G$ consider the $W_GH$ module $M(G/H)/Im(Tr)$ where $W_GH=N_GH/H$ is the Weyl group and $Im(Tr)$ is the submodule spanned by the images of all transfer maps $Tr_K^H$ for $K\subseteq H$.  If we let $H$ vary over representatives of conjugacy classes of subgroups of $G$ then we get a sequence of $W_GH$ modules.	This functor from rational $G$-Mackey functors to sequences of $\Q[W_GH]$-modules is an equivalence of symmetric monoidal categories. 
	\end{itemize}
	
	From now on, we specialize to the case $G=C_{2^n}$.\medbreak
	
	There are two $1$-dimensional $\Q[G]$ modules up to isomorphism: $\Q$ with the trivial action and $\Q$ with action $g\cdot 1=-1$ where $g\in G$ is a generator. We shall denote the two modules by $\Q$ and $\Q_-$ respectively; every other module splits into a sum of these.

	The representatives of conjugacy classes of $G=C_{2^n}$ are $H=C_{2^i}$ for every $0\le i\le n$ thus the datum of a rational $G$-Mackey functor is equivalent to a sequence of rational $W_GH=C_{2^n}/C_{2^i}$ modules.

	We let $M_i^+, 0\le i\le n$, and $M_i^-, 0\le i<n$, be the Mackey functors corresponding to the sequences $C_{2^n}/C_{2^j}\mapsto \delta_{ij}\Q$ and $C_{2^n}/C_{2^j}\mapsto \delta_{ij}\Q_-$ respectively.
	
	For example, $M_0^+,M_0^-$ are the constant Mackey functors corresponding to the modules $\Q$ and $\Q_-$ respectively. 
	
	Observe that:
	\begin{itemize}
		\item The $M_i^{\pm}$ are self-dual.
		\item $M_i^{\pm}\boxtimes M_j^{\pm}=0$ if $i\neq j$.
		\item $M_i^{\alpha}\boxtimes M_i^{\beta}=M_i^{\alpha\beta}$ where $\alpha,\beta\in \{-1,1\}$.
	\end{itemize}
	Henceforth we shall write $M_i$ for $M_i^+$.\medbreak
	
	The notation $M_i\{a\}$ shall mean a copy of $M_i$ with a choice of generator $a\in M_i(C_{2^n}/C_{2^i})=\Q$. The element $a$ generates $M_i\{a\}$ through its transfers:
\begin{equation}M_i\{a\}(C_{2^n}/C_{2^j})=\begin{cases}\Q\{\Tr_{2^i}^{2^j}(a)\}&\textup{if, }j\ge i\\
0&\textup{if, }j< i
	\end{cases}
	\end{equation}
We analogously define $M_i^-\{a\}$.\medbreak
	
The rational Burnside $G$-Green functor is
	\begin{equation}
	A_{\Q}(C_{2^n}/C_{2^i})=\frac{\Q[x_{i,j}]}{x_{i,j}x_{i,k}=2^{i-\max(j,k)}x_{i,\min(j,k)}}
	\end{equation} 
	where $x_{i,j}=[C_{2^i}/C_{2^j}]\in A(C_{2^i})$ for $0\le j< i$. To complete the Mackey functor description, we note that:
	\begin{gather}
	\Tr_{2^i}^{2^{i+1}}(x_{i,j})=x_{i+1,j}\text{ , }\Tr_{2^i}^{2^{i+1}}(1)=x_{i+1,i}	\end{gather}
	Let
	\begin{equation}
	y_i=\begin{cases}
	1-\frac{x_{i,i-1}}2&\textup{if, }i\ge 1\\
	1&\textup{if, }i=0
	\end{cases}
	\end{equation}
	living in $A_{\Q}(C_{2^n}/C_{2^i})$. We can see that $y_i$ spans a copy of $M_i$ in $A_{\Q}$ and:
	\begin{equation}
	A_{\Q}=\oplus_{i=0}^nM_i\{y_i\}
	\end{equation}
	This is an isomorphism of Green functors, where the RHS becomes a Green functor by setting the product of elements from different summands to be $0$ and furthermore setting the $y_i$ to be idempotent ($y_i^2=y_i$).\medbreak
	
	
	\section{Euler and orientation classes}\label{EulerAndOr}
The real representation ring	$RO(C_{2^n})$ is spanned by the irreducible representations $1,\sigma, \lambda_{s,k}$ where $\sigma$ is the $1$-dimensional sign representation and $\lambda_{s,m}$ is the $2$-dimensional representation given by rotation by $2\pi s (m/2^n)$ degrees for $1\le m$ dividing $2^{n-2}$ and odd $1\le s<2^n/m$. Note that $2$-locally, $S^{\lambda_{s,m}}\simeq S^{\lambda_{1,m}}$ as $C_{2^n}$-equivariant spaces, by the $s$-power map. Therefore, to compute $H_{\bigstar}^{C_{2^n}}(X)$ it suffices to only consider $\bigstar$ in the span of $1,\sigma,\lambda_k:=\lambda_{1,2^k}$ for $0\le k\le n-2$ ($\lambda_{n-1}=2\sigma$ and $\lambda_n=2$).\medbreak
	
	We shall now define generating classes for $H_{\bigstar}^G$.
	
	We first have Euler classes $a_{\sigma}:S^0\hookrightarrow S^{\sigma}$ and $a_{\lambda_k}:S^0\hookrightarrow S^{\lambda_k}$ given by the inclusion of the north and south poles; under the Hurewicz map these classes are $a_{\sigma}\in H_{-\sigma}^G(G/G)$ and $a_{\lambda_k}\in H_{-\lambda_k}^G(G/G)$.
	
	There are also orientation classes $u_{\sigma}\in H_{1-\sigma}^G(C_{2^n}/C_{2^{n-1}})$, $u_{2\sigma}\in H^G_{2-2\sigma}(G/G)$ and $u_{\lambda_k}\in H^G_{2-\lambda_k}(G/G)$ but we shall need a small computation in order to define them.
	
	Using the cofiber sequence $C_{2^n}/C_{2^{n-1}+}\to S^0\xrightarrow{a_{\sigma}} S^{\sigma}$ we get:
	\begin{gather}
\tilde	H_0^G(S^{\sigma})=M_n\{a_{\sigma}\}\\
\tilde	H_1^G(S^{\sigma})=\oplus_{i=0}^{n-1}M_i^-
	\end{gather}
	where $\tilde H_*^G(X)$ denotes the reduced homology of a based $G$-space $X$. 
	We can further see that $\tilde H_1^G(S^{\sigma})$ is generated as a Green functor module by a class $u_{\sigma}\in \tilde H_{1-\sigma}^G(C_{2^n}/C_{2^{n-1}})$. So we get
	\begin{equation}
	\tilde H_*^G(S^{\sigma})=	M_n\{a_{\sigma}\}\oplus \oplus_{i=0}^{n-1}M_i^{-}\{y_i\Res_{2^i}^{2^{n-1}}(u_{\sigma})\}
	\end{equation}
	Using that $S^{2\sigma}=S^{\sigma}\wedge S^{\sigma}$ and the Kunneth formula, we get a class $u_{2\sigma}$ restricting to $u_{\sigma}^2$ and:
	\begin{gather}
\tilde	H_*^G(S^{2\sigma})=	M_n\{a_{\sigma}^2\}\oplus \oplus_{i=0}^{n-1}M_i\{y_i\Res_{2^i}^{2^{n}}(u_{2\sigma})\}
	\end{gather}
	
	For $0\le k\le n-2$ we have a $G$-CW decomposition $S^0\subseteq X\subseteq S^{\lambda_k}$ where $X$ consists of the points $(x_1,x_2,x_3)\in S^{\lambda_k}\subseteq \R^3$ with $x_1=0$ or $x_2=0$. From this decomposition we can see that:
	\begin{gather}
\tilde	H_0^G(S^{\lambda_k})=\oplus_{i=k+1}^{n}M_i\{y_i\Res_{2^i}^{2^n}(a_{\lambda_k})\}	\\
\tilde	H_2^G(S^{\lambda_k})=\oplus_{i=0}^{k}M_i\{y_i\Res_{2^i}^{2^n}(u_{\lambda_k})\} 
	\end{gather}
	for a class $u_{\lambda_k}\in H^G_{2-\lambda_k}(G/G)$. This also works for $k=n-1$ and $\lambda_{n-1}=2\sigma$ giving a different way of obtaining $a_{2\sigma}=a_{\sigma}^2$ and $u_{2\sigma}$.\medbreak
	
	The classes $u_{\sigma}, u_{\lambda_k}$, $0\le k\le n-1$, have not been canonically defined so far. Once we fix orientations for the spheres $S^{\lambda_k}$, the $u_{\lambda_k}$ are uniquely determined by the following two facts:
	\begin{itemize}
		\item A $G$-self-equivalence of $S^{\lambda_k}$ induces the identity map on the Mackey functor $\tilde H^G_{2}(S^{\lambda_k})$ if it does so on its bottom level $\tilde H^G_2(S^{\lambda_k})(G/e)$.
		\item An orientation of $S^{\lambda_k}$ determines a generator for $\Z=\tilde H_2(S^2;\Z)$ and consequently a generator for $\Q=\tilde H_2(S^2;\Q)=\tilde H_2^G(S^{\lambda_k})(G/e)$.
	\end{itemize}
The first fact is proven using that the Mackey functor $\tilde H^G_{2}(S^{\lambda_k})$ is generated by the transfers of $y_i\Res^{2^n}_{2^i}(u_{\lambda_k})$ where $i\le k$, so we only need to check that the induced map is the identity on $\tilde H^G_{2}(S^{\lambda_k})(G/C_{2^i})=\tilde H^{C_{2^i}}_2(S^{\lambda_k})(C_{2^i}/C_{2^i})$ which follows from the fact that $C_{2^i}$ acts trivially on $S^{\lambda_k}$ when $i\le k$.\medbreak

We can similarly uniquely determine $u_{\sigma}$ upon fixing an orientation of $S^{\sigma}$ that is compatible with the orientation for $S^{\lambda_{n-1}}=S^{2\sigma}$, meaning that $\Res^{2^n}_{2^{n-1}}(u_{2\sigma})=u_{\sigma}^2$.\medbreak

The discussion regarding orientation classes can also be performed integrally, defining $A_{\Z}$-orientation classes $u_{\sigma},u_{2\sigma},u_{\lambda_k}$ upon fixing orientations for  $S^{\sigma},S^{2\sigma}$, $S^{\lambda_k}$ as above. The $A_{\Z}$-orientation classes map to the corresponding $\Z$-orientation classes of \cite{HHR16} under the map $HA_{\Z}\to H\underline{\Z}$ where $\underline{\Z}$ is the constant Green functor corresponding to the trivial $G$-module $\Z$. 

	
	\section{Rational stable stems}\label{RationalStems}
	
In this section we shall give a presentation of the Green functor $H^G_{\bigstar}$ with generators and relations. 

The generators are elements $r_k\in H^G_{V_k}(C_{2^n}/C_{2^{i_k}})$ spanning $M_{i_k}^{\epsilon_k}$, where $\epsilon_k=+$ or $-$, such that every element of $\coprod_{H\subseteq G, \bigstar\in RO(G)}H_{\bigstar}^G(G/H)$  can be obtained from the $r_k$ using the operations of addition, multiplication, restriction, transfer and scalar multiplication (where the scalars are elements of $\coprod_{H\subseteq G}A_{\Q}(G/H)$).

The fact that the $r_k$ span $M_{i_k}^{\epsilon_k}$ gives all the additive (Mackey functor) relations, but also implies certain multiplicative relations by means of the Kunneth formula: If $i_k< i_l$ then $r_k\cdot \Res^{2^{i_l}}_{2^{i_k}}(r_l)=0$ and if $i_k=i_l$ then $r_kr_l$ spans $M_{i_k}^{\epsilon_k\epsilon_l}$.

Finally, if $r\in H^G_{V_k}(C_{2^n}/C_{2^i})$ and there exists a unique $r'\in H^G_{-V_k}(C_{2^n}/C_{2^i})$ with $rr'=y_i$, then we shall use the notation $y_i/r$ to denote $r'$. If $r,y_i/r$ are generators then we have the implicit relation $r\cdot (y_i/r)=y_i$.

\begin{prop}\label{General}The Green functor $H_{\bigstar}^G$ has a presentation whose generating set  is the union of the following four families:
	\begin{itemize}\item $y_i\Res_{2^i}^{2^{n-1}}(u_{\sigma})$ and $y_i/\Res_{2^i}^{2^{n-1}}(u_{\sigma})$ spanning $M_i^{-}$, where $0\le i<n$.\smallbreak
	\item $y_i\Res^{2^n}_{2^i}(u_{\lambda_k})$ and $y_i/\Res^{2^n}_{2^i}(u_{\lambda_k})$ spanning $M_i$, where $0\le i\le k$ and $0\le k\le n-2$.\smallbreak
	\item $y_i\Res^{2^n}_{2^i}(a_{\lambda_{k}})$ and $ y_i/\Res^{2^n}_{2^i}(a_{\lambda_{k}})$ spanning $M_i$, where $k< i\le n$ and $0\le k\le n-2$.\smallbreak
	\item $a_{\sigma}$ ($=y_na_{\sigma}$) and $y_n/a_{\sigma}$ spanning $M_n$.
	\end{itemize}
	We have implicit relations of the form $(y_i\gamma)\cdot (y_i/\gamma)=y_i$ in each of the four families. The remaining multiplicative relations can be obtained using the Kunneth formula.
	\end{prop}

Two observations:
\begin{itemize}\item For $0\le i<n$, the square of $y_i\Res_{2^i}^{2^{n-1}}(u_{\sigma})$ is $y_i\Res_{2^i}^{2^n}(u_{2\sigma})$ and spans $M_i$.  
\item  The ring $H_{\bigstar}^G(G/G)$ has multiplicative relations:  $a_{\sigma}u_{2\sigma}=0$,  $a_{\sigma}u_{\lambda_k}=0$ and $a_{\lambda_k}u_{\lambda_s}=0$ for $s\le k$.
\end{itemize}

The Green functor presentation also gives us an additive decomposition of $H^G_{\bigstar}$ into $M_i,M_i^-$ but to state it explicitly, we'll need some notation: For each integer tuple $t=(j_0,...,j_{n-1},j'_0,...,j'_{n-1})$ let
	\begin{equation}
	k(t)=\begin{cases}
	n&\textup{ if $j_k=0$ for all $k$}\\
	\min\{k:j_k\neq 0\}&\textup{ otherwise}
	\end{cases}
	\end{equation}
	and
		\begin{equation}
	k'(t)=\begin{cases}
	-1&\textup{ if $j'_{k'}=0$ for all $k'$}\\
	\max\{k':j'_{k'}\neq 0\}&\textup{ otherwise}
	\end{cases}
	\end{equation}
and consider the representation
	\begin{gather}
V_t^{\pm}=	\sum_{k=0}^{n-2}(j_k(2-\lambda_k)-j'_k\lambda_k)+j_{n-1}(1-\sigma)-j'_{n-1}\sigma
	\end{gather} 
	where the sign $\pm$ in $V_t^{\pm}$ is $+$ if $j_{n-1}$ is even and $-$ if $j_{n-1}$ is odd.
	
	Let $T$ be the set of all tuples $t$ with $k'(t)<k(t)$; as $t$ ranges over $T$, the $V_t^{\pm}$ are pairwise non-isomorphic virtual representations. We can now state the additive description:
		\begin{prop}The $C_{2^n}$ equivariant rational stable stems are:
			\begin{equation}
			H^G_{\bigstar}=\begin{cases}
			\oplus_{k'(t)<i\le k(t)}M_i&\textup{if }\bigstar=V_t^+\text{ for }t\in T\\
			\oplus_{k'(t)<i\le k(t)}M_i^-&\textup{if }\bigstar=V_t^-\text{ for }t\in T\\
			0&\textup{otherwise }
			\end{cases}
			\end{equation}
				\end{prop}
			
\begin{proof}(Of Proposition \ref{General}). Any representation sphere $S^V$ is the smash product of $S^{\sigma},S^{\lambda_k}$ and their duals $S^{-\sigma},S^{-\lambda_k}$. 	 
	By duality, 
	\begin{equation}
	\tilde H_*^G(S^{-\sigma})=	\tilde H^{-*}_G(S^{\sigma})=	M_n\oplus \oplus_{i=0}^{n-1}M_i^{-}\{y_i\Res_{2^i}^{2^{n-1}}(u_{\sigma}^{-1})\}
	\end{equation}
	Let $t$ be a generator for this copy of $M_n$; then
	\begin{equation}
	\tilde H_0^G(S^0)=\tilde H_0^G(S^{\sigma})\boxtimes \tilde H_0^G(S^{-\sigma})\oplus \tilde H_1^G(S^{\sigma})\boxtimes \tilde H_{-1}^G(S^{-\sigma})
	\end{equation}
	On the left hand side we have a factor $M_n\{y_n\}$ and on the right hand side we have $M_n\{a_{\sigma}\}\boxtimes M_n\{t\}=M_n\{a_{\sigma}t\}$ so $y_n=\lambda a_{\sigma}t$ for $\lambda\in \Q^{\times}$. Thus we can pick $t=y_n/a_{\sigma}$.
	The result then follows from the Kunneth formula.
\end{proof}

	We note that taking geometric fixed points inverts all Euler classes, annihilating all orientation classes and setting $y_i=1$. Therefore:
	\begin{equation}\Phi^{C_{2^n}}(HA_{\Q})_{\bigstar}=\Q[a_{\sigma}^{\pm},a_{\lambda_k}^{\pm}]_{0\le k\le n-2}
	\end{equation}
	hence $\Phi^{C_{2^n}}HA_{\Q}=H\Q$ as nonequivariant spectra. The homotopy fixed points, homotopy orbits and Tate fixed points are computed using that $HA_{\Q}\to H\underline{\Q}$ is a nonequivariant equivalence, where $\underline \Q=M_0$ is the constant Green functor. Thus:
	\begin{equation}(HA_{\Q})_{hC_{2^n}\bigstar}=(HA_{\Q})^{hC_{2^n}}_{\bigstar}=\Q[u_{2\sigma}^{\pm},u_{\lambda_k}^{\pm}]_{0\le k\le n-2}
	\end{equation}
	and $(HA_{\Q})^{tC_{2^n}}=*$.

	\section{\texorpdfstring{$C_{2^n}$ rational characteristic classes}{C2n characteristic classes}}\label{RationalCharacteristic}

	\begin{prop}\label{C2nChern1Class} As a Green functor algebra over the homology of a point:
					\begin{equation}
		H_G^{\bigstar}(B_GS^1)=\frac{H_G^{\bigstar}[\genw,\genu_{m,j}]_{1\le m\le n, 1\le j<2^m}}{\genu_{m,j}\genu_{m',j'}=\delta_{mm'}\delta_{jj'}\genu_{m,j}\text{ , }\Res^{2^n}_{2^{m-1}}(\genu_{m,j})=0}
		\end{equation}
		for $|\genw|=2$ and $|\genu_{m,j}|=0$.
	\end{prop}

	\begin{proof}Note that
		\begin{equation}
		H^{\bigstar}_G(X)=H_G^*(X)\boxtimes_{A_{\Q}}H^G_{\bigstar}
		\end{equation}
		so it suffices to describe the integer graded cohomology.
		
		For an explicit model of $B_GS^1$ we take $\C P^{\infty}$ with a $G$ action that can be described as follows: Let $V_1,...,V_{2^n}$ be an ordering on the irreducible complex $G$-representations and set $V_{k+2^nm}=V_k$ for any $m\in \Z$, $1\le k\le 2^n$. The action of $g\in G$ on homogeneous coordinates is $g(z_1:z_2:\cdots)=(gz_1:gz_2:\cdots)$ where $g$ acts on $z_i$ as it does on $V_i$.
		
		Fix a subgroup $H=C_{2^m}$ of $G$. The fixed points under the $H$-action are:
		\begin{equation}
		(B_GS^1)^H=\coprod_{j=1}^{2^m}\C P^{\infty}
		\end{equation}
		To understand the indexing, let $W_1,...,W_{2^m}$ be an ordering on the irreducible complex $C_{2^m}$-representations; the $j$-th $\C P^{\infty}$ in $(B_GS^1)^H$ corresponds to the set of points with homogeneous coordinates $(z_1:z_2:\cdots)$ such that $z_k=0$ if $\Res_{2^m}^{2^n}(V_k)\neq W_j$.
		
		By \cite{GM95} we have:
		\begin{equation}
		H^*_G(B_GS^1)=\oplus_{m=0}^n H^*((B_GS^1)^{C_{2^m}})^{C_{2^n}/C_{2^m}}
		\end{equation}
		where $H^*(X)$ is nonequivariant cohomology in $\Q$ coefficients. The action of $C_{2^n}/C_{2^m}$ on nonequivariant cohomology is trivial since it's determined in degree $*=2$ and thus on the $2$-skeleton, which itself is the disjoint union of copies of $S^2=\C P^1$ and for each $S^2$ the action is a rotation hence has degree $1$. Thus
				\begin{equation}
		H^*_G(B_GS^1)=\oplus_{m=0}^n\oplus^{2^m}_{j=1} H^*(\C P^{\infty})=\oplus_{m=0}^n\oplus^{2^m}_{j=1}\Q[\geneidem_{m,j}]
		\end{equation}
		where each $\geneidem_{m,j}$ spans $M_m$. Set $\genu_{m,j}=\geneidem_{m,j}^0$ and $\genw=\sum_{m,j}\geneidem_{m,j}$; then
		\begin{equation}
		\sum_{j=1}^{2^m}\genu_{m,j}=\frac{\Tr_{2^m}^{2^n}(y_m)}{2^m}
		\end{equation}
		so the $\genu_{m,2^m}$ are superfluous. Thus we can take $1\le m\le n$ and $1\le j<2^m$ in the indexing for $\genu_{m,j}$.
	\end{proof}

We can similarly prove that:
	\begin{prop}We have an isomorphism of Green functor algebras over $H^{\bigstar}_G$:
	\begin{equation}
		H_G^{\bigstar}(B_G\Sigma_2)=\frac{H^{\bigstar}_G(B_GS^1)}{\genw}
	\end{equation}
	where the quotient map $H^{\bigstar}_G(B_GS^1)\to H_G^{\bigstar}(B_G\Sigma_2)$ is induced by complexification: $B_G\Sigma_2=B_GO(1)\to B_GU(1)=B_GS^1$.
\end{prop}

The set of generators $\{\genw,\genu_{m,j}\}$ for $H^*_G(B_GS^1)$ is not minimal. Indeed, whenever we have generators $e_1,...,e_s$ with $e_ie_j=\delta_{ij}e_i$, we can replace them by a single generator defined by $e=e_1+2e_2+\cdots+se_s$:
	\begin{gather}
\frac{\Q[e_1,...,e_s]}{e_ie_j=\delta_{ij}e_i}=\frac{\Q[e]}{e(e-1)\cdots (e-s)}
\end{gather}
This isomorphism follows from the fact that any polynomial $f$ on $e_1,...,e_n$ satisfies:
\begin{equation}
f(e)=f(0)+(f(1)-f(0))e_1+\cdots +(f(s)-f(0))e_s
\end{equation}
and thus
\begin{equation}
e_i=\frac{f_i(e)}{f_i(i)}\text{ where }f_i(x)=\frac{x(x-1)\cdots (x-s)}{x-i}
\end{equation}

In this way, $H^*_G(B_GS^1)$ is generated as an $A_{\Q}$ algebra by two elements $\genw,\genu$ but now with $\genu$ satisfying some rather complicated relations. If $n=1$ i.e. $G=C_{2}$, we only have one $\genu_{m,j}$ element, namely $\genu=\genu_{1,1}$ satisfying $\genu^2=\genu$.

\begin{prop}The inclusion $B_GU(1)^m\to B_GU(m)$ induces an isomorphism of Green functor algebras over $H^{\bigstar}_G$:
	\begin{equation}
		H^{\bigstar}_G(B_GU(m))=(\otimes^mH^{\bigstar}_G(B_GU(1)))^{\Sigma_m}
	\end{equation}
\end{prop}

\begin{proof}Let $V_i$ be the complex $G$-representation corresponding to the root of unity $e^{2\pi i/n}$. The Grassmannian model for $B_GU(m)$ uses complex $m$-dimensional subspaces of $\C^{\infty \rho_G}$; a $G$-fixed point $W$ of $B_GU(m)$ is then a $G$-representation and thus as splits as $W=\oplus_{i=1}^{2^n}k_iV_i$ for $k_i=0,1,...$ with $\sum_ik_i=m$. An automorphism of $W$ is made out automorphisms for each $k_iV_i$ hence
	\begin{equation}
		B_GU(m)^{G}=\coprod_{\sum k_i=m}\prod_{i=1}^{2^n}BU(k_i)
	\end{equation}
Following \cite{Rat1} and inducting on the $n$ in $G=C_{2^n}$, it suffices to show that 
	\begin{equation}
	H^*((B_GU(m))^G)\to H^*(\prod^m(B_GS^1)^G)
\end{equation}
	is an isomorphism after taking $\Sigma_m$ fixed points on the RHS. Spelling this out, we have:
		\begin{equation}
		\prod_{\sum k_i=m}\otimes_{i=1}^{2^n}H^*(BU(k_i))\to \prod^{2^{nm}}\otimes^m H^*(BS^1)
	\end{equation}
where the product on the right is indexed on configurations $(V_{r_1},...,V_{r_m})$. If we fix $k_i$ with $\sum_ik_i=m$ then we get
		\begin{equation}\otimes_{i=1}^{2^n}H^*(BU(k_i))\to \prod\otimes^m H^*(BS^1)
\end{equation}
where the product on the right is indexed on configurations $(V_{r_1},...,V_{r_m})$ where $k_i$ many of the $r_j$'s are equal to $i$. Taking $\Sigma_m$ fixed points is equivalent to fixing a configuration and then taking $\Sigma_{k_1}\times \cdots\times \Sigma_{k_{2^n}}$ fixed points, where each $\Sigma_{k_i}$ permutes the $k_i$ many coordinates that are $V_i$ in the configuration. Thus we are reduced to the nonequivariant isomorphism:
		\begin{equation}H^*(BU(k_i))=(\otimes^{k_i} H^*(BS^1))^{\Sigma_{k_i}}
\end{equation}
\end{proof}

For $n=1$, $G=C_2$ and $H^*_G(B_GS^1)$ has a simple enough description to allow the computation of an explicit minimal presentation of $H^*_G(B_GU(m))=H^*_G(B_GU(1)^m)^{\Sigma_m}$. Due to the greater algebraic complexity of $H^*_G(B_GS^1)$ for $n\ge 2$ ($G=C_{2^n}$), we do not attempt to generalize this and the rest of \cite{Rat1} to groups $G=C_{2^n}$ for $n\ge 2$.\medbreak

We note that the maximal torus isomorphism does not work $C_{2^n}$ equivariantly for the Lie group $L=SU(2)=Sp(1)$ and $n\ge 2$. The reason is that a $C_{2^n}$ representation in $SU(2)$ is $2, 2\sigma$ or $V_i\oplus V_{-i}$, $1\le i< 2^{n-1}$, using the notation of the proof above. Thus:
\begin{equation}
	B_GSU(2)^G=BSU(2)\coprod BSU(2)\coprod\coprod_{i=1}^{2^{n-1}-1} BS^1
\end{equation}
so $H^0(B_GSU(2)^G)$ has dimension $2^{n-1}+1$. On the other hand, the maximal torus is $U(1)\subseteq SU(2)$ with Weyl group $C_2$ and
\begin{equation}
	B_GU(1)^G=\coprod_{i=1}^{2^n} BS^1
\end{equation}
The $C_2$ action does not affect $H^0(B_GU(1)^G)$ which has dimension $2^n$. Finally, $2^{n-1}+1=2^n$ only when $n=1$.

\phantom{1}\smallbreak

\begin{small}
	\noindent  \textsc{Department of Mathematics, University of Chicago}\\
	\textit{E-mail:} \verb|nickg@math.uchicago.edu|\\
	\textit{Website:} \href{http:://math.uchicago.edu/~nickg}{math.uchicago.edu/$\sim$nickg}
\end{small}


\begin{thebibliography}{999}

\bibitem[Geo21c]{Rat1}N. Georgakopoulos, \emph{$C_2$ equivariant characteristic classes over the rational Burnside ring}, available	\href{https://math.uchicago.edu/~nickg/papers/Rational.pdf}{here}\smallbreak

\bibitem[GM95]{GM95} J.P.C. Greenlees, J.P. May, \emph{Generalized Tate Cohomology}, Memoirs of the American Mathematical Society
543 (1995).	\smallbreak	

\bibitem[HHR16]{HHR16}	M. A. Hill, M. J. Hopkins, D. C. Ravenel, \emph{On the non-existence
	of elements of Kervaire invariant one}, Annals of Mathematics, Volume 184 (2016), Issue 1
\smallbreak

\end{thebibliography}
\end{document}